\documentclass[a4paper]{amsart}
\usepackage{amsmath, amsthm, amsfonts, amscd, amssymb, MnSymbol}
\pdfoutput=1
\usepackage{xcolor}
\usepackage{enumitem}
\usepackage{tikz-cd}
\usepackage[normalem]{ulem}

\newtheorem{thm}{Theorem}[section]
\newtheorem{cor}[thm]{Corollary}
\newtheorem{defn}[thm]{Definition}
\newtheorem{lem}[thm]{Lemma}
\newtheorem{rmk}[thm]{Remark}
\newtheorem{prop}[thm]{Proposition}

{\theoremstyle{remark}}

\numberwithin{thm}{section}
\numberwithin{equation}{section}

\usepackage{biblatex}
\usepackage{hyperref}
\usepackage{xurl}
\hypersetup{breaklinks=true}

\newcommand{\id}{\operatorname{id}}

\newcommand{\Ad}{\operatorname{Ad}}

\newcommand{\Spec}{\operatorname{Spec}}
\newcommand{\End}{\operatorname{End}}
\newcommand{\Irr}{\operatorname{Irr}}
\newcommand{\SSp}{*\text{-}\operatorname{Irr}}

\newcommand{\Nbb}{\mathbb N}
\newcommand{\Cbb}{\mathbb C}

\newcommand{\Rbb}{\mathbb R}

\newcommand{\Lbb}{\mathbb{L}}

\newcommand{\Dcal}{\mathcal{D}}

\newcommand{\Hcal}{\mathcal{H}}

\newcommand{\Mcal}{\mathcal{M}}

\newcommand{\gf}{\mathfrak{g}}
\newcommand{\hf}{\mathfrak{h}}
\newcommand{\kf}{\mathfrak{k}}

\newcommand{\Cf}{\mathfrak{C}}
\newcommand{\tf}{\mathfrak{t}}

\newcommand{\la}{\langle}
\newcommand{\ra}{\rangle}

\newcommand{\ww}{\mathrm{W}}

\newcommand{\roots}{\mathbf{Q}}

\newcommand{\weights}{\mathbf{P}}

\DeclareMathOperator{\Dom}{Dom}

\begin{document}

\title[Beurling-Fourier algebras of $ q $-deformations]{Beurling-Fourier algebras of $ q $-deformations of compact semisimple Lie groups and complexification}

\author{Heon Lee}
\address{Department of Mathematical Sciences and the Research Institute of Mathematics \\
         Seoul National University \\ 
				 Gwanak-ro 1, Gwanak-gu \\
				 Seoul 08826 \\ 
				 Republic of Korea
}
\email{heoney93@gmail.com}

\author{Christian Voigt} 
\address{School of Mathematics and Statistics \\
         University of Glasgow \\
         University Place \\
         Glasgow G12 8QQ \\
         United Kingdom 
}
\email{christian.voigt@glasgow.ac.uk}


\begin{abstract}
We study Beurling-Fourier algebras of $ q $-deformations of compact semisimple Lie groups. In particular, we show that the space of irreducible representations 
of the function algebras of their Drinfeld doubles is exhausted by the irreducible representations of weighted Fourier algebras associated to a certain family of central weights. 
\end{abstract} 

\maketitle

\section{Introduction} \label{sec:introduction}

Beurling-Fourier algebras are weighted versions of the Fourier algebra of a locally compact group \cite{Ludwig2012}. For Lie groups, the spectra of these commutative Banach algebras are closely related to the complexification of the underlying group \cite{Ghandehari2021, Giselsson2024}. This provides an interesting link between harmonic analysis and geometry. 

The definition of weighted Fourier algebras makes also sense for locally compact quantum groups in the sense of Kustermans and Vaes \cite{KVLCQG}. For compact quantum groups, this was studied by Franz and Lee in \cite{Franz2021}, where it was shown that the resulting theory is again closely linked with complexification. Due to noncommutativity it is however no longer suitable to consider only characters. Instead, at least in the type I situation, one should work with arbitrary irreducible representations.  
For Woronowicz's quantum group $ SU_q(2) $, Franz and Lee determined a concrete family of symmetric central weights such that the representations of the associated Beurling-Fourier algebras exhaust the space of irreducible representations of the quantum Lorentz group \cite{Podles1990}. The latter is defined as the Drinfeld double of $ SU_q(2) $, and can be interpreted as the complexification of $ SU_q(2) $ via the quantum duality principle \cite{Drinfeld}, \cite{Gavarini}. 

In this paper we extend these considerations to $ q $-deformations of general compact semisimple Lie groups. For such a quantum group $ K_q $, we exhibit a natural family of weights on the dual of $ K_q $ such that the associated Beurling-Fourier algebras detect the irreducible representations of the function algebra of the Drinfeld double of $ K_q $, and also obtain precise estimates on the which Hilbert space representations of the polynomial function algebra extend to completely bounded representations of the Beurling-Fourier algebras. 

In view of the above, one may hope to explore the concept of complexification for more general classes of quantum groups using Beurling-Fourier algebras. However, we will not address this problem here.  

Let us conclude with some remarks on notation. We write $ \mathbb{N} = \{0,1,2,\dots \} $ for the set of natural numbers. The inner product of a Hilbert space will always 
be denoted by $ \la \cdot, \cdot \ra $, with the first argument being conjugate linear. Bilinear pairings between vector spaces, such as the canonical pairing between a 
vector space and its dual, will be denoted by $ (\cdot , \cdot) $, and the context will make it clear which pairing is being used.
The algebra of bounded operators on a Hilbert space $ \Hcal $ will be denoted by $ \Lbb(\Hcal) $. Given a $ * $-representation $ \pi $ of an involutive algebra on a Hilbert space, we 
typically denote the representation space by $ \Hcal_\pi $.
The spatial tensor product of von Neumann algebras is denoted by $ \overline{\otimes} $.  
We write $ \otimes $ for the algebraic tensor product of vector spaces, the Hilbert space tensor product, and the minimal tensor product of $ C^* $-algebras. 
In all cases, the context should make it clear which tensor product is being referred to.

We would like to thank Hun Hee Lee for discussions on the subject of this paper.

\section{Preliminaries} \label{sec:preliminaries} 

In this section we collect some background material and fix our notation. 

\subsection{Locally compact quantum groups} \label{sec:Locally compact quantum groups}

The theory of locally compact quantum groups was developed by Kustermans and Vaes \cite{KVLCQG}. We shall briefly review some key facts. 

By definition, a locally compact quantum group $ G $ is given by a von Neumann algebra $ L^\infty(G) $ together with a normal 
unital $ * $-homomorphism $ \Delta: L^\infty(G) \rightarrow L^\infty(G) \overline{\otimes} L^\infty(G) $ called comultiplication, 
satisfying $ (\Delta \otimes \id) \Delta = (\id \otimes \Delta) \Delta $, 
and normal semifinite faithful weights $ \varphi, \psi $ on $ L^\infty(G) $, called left and right Haar weights, respectively. These 
satisfy certain invariance conditions with respect to $ \Delta $, 
and are unique up to a scalar. We note that the von Neumann algebra $ L^\infty(G) $ is typically non-commutative, and not an algebra of function on a measure space. We write $ L^2(G) $ for the GNS Hilbert space of $ \varphi $, and $ \Lambda: \mathcal{N}_\varphi \rightarrow L^2(G) $ for the GNS map, where $ \mathcal{N}_\varphi = \{f \in L^\infty(G) \mid \varphi(f^*f) < \infty\} $. 

With any locally compact quantum group $ G $ one can associate the dual locally compact quantum group $ \widehat{G} $ in such a way that the correspondence 
between $ G $ and $ \widehat{G} $ extends Pontryagin duality. We will refer to the von Neumann algebra $ L^\infty(\widehat{G}) = VN(G) $ as the group von Neumann algebra of $ G $. The Hilbert spaces $ L^2(G) $ and $ L^2(\widehat{G})$ can be identified in a canonical way.

The \emph{Kac-Takesaki operator} $ \ww \in L^\infty (G) \overline{\otimes} VN(G) $ is the operator on $ L^2(G) \otimes L^2(G) $ defined via
\[
((\omega\otimes\id)\ww^*) \Lambda(f) = \Lambda((\omega\otimes \id)\Delta(f)). 
\]
It is unitary and satisfies the \emph{pentagon equation} 
\[
\ww_{12} \ww_{13} \ww_{23} = \ww_{23} \ww_{12}
\]
in $ \Lbb(L^2(G) \otimes L^2(G) \otimes L^2(G)) $, where we are using the leg-numbering notation. Moreover it implements the comultiplication via $ \Delta(f)=\ww^*(\id \otimes f)\ww $ for $ f \in L^\infty(G) $. 
By duality we also get the Kac-Takesaki operator $ \widehat{\ww} $ for $ \widehat{G} $, which 
is linked to $\ww $ via $ \widehat{\ww} = \Sigma \ww^* \Sigma $, where $ \Sigma $ is the tensor flip. 

The \emph{antipode} $ S $ is a densely defined, typically unbounded, operator on $ L^\infty(G) $ such that
\[
(\id\otimes \omega)\ww \in \Dom(S) \text{ and } S((\id\otimes\omega)\ww) = (\id\otimes \omega)(\ww^*),  
\]
and the \emph{unitary antipode} $ R $ is a normal $ * $-preserving antimultiplicative map on $ L^\infty(G) $ 
satisfying $ \Delta R = \sigma (R \otimes R) \Delta $, where $ \sigma $ denotes the tensor flip. 
These maps are linked via $ S = R \tau_{-i/2} = \tau_{-i/2} R $, where $ (\tau_t)_{t \in \mathbb{R}} $ is the group of \emph{scaling automorphisms} of $ L^\infty(G) $. 
We similarly obtain the antipode $ \widehat{S} $, the unitary antipode $ \widehat{R} $, and the scaling group $ (\widehat{\tau}_t)_{t \in \mathbb{R}} $ for $ L^\infty(\widehat{G}) $. 

Let $ L^1(\widehat{G}) $ be the predual of the von Neumann algebra $ L^\infty(\widehat{G}) = VN(G) $. Since $ L^\infty(\widehat{G}) \subset \Lbb(L^2(G)) $ is in standard form we have 
\[
L^1(\widehat{G}) = \{\omega_{\xi, \eta} \mid \xi, \eta \in L^2(G) \}, 
\]
where $ \omega_{\xi, \eta}(T) = \langle \xi, T \eta \rangle  $ is the vector functional associated with $ \xi, \eta $. 
The map $ \widehat{\lambda}: L^1(\widehat{G}) \rightarrow L^\infty(G), \widehat{\lambda}(\omega) = (\omega \otimes \id)(\widehat{\ww}) $ is injective and, by definition, the \emph{Fourier algebra} of $ G $ is the image $ A(G) = \widehat{\lambda}(L^1(\widehat{G})) $ with the norm induced from $ L^1(\widehat{G}) $. Together with the multiplication from $ L^\infty(G) $, the Fourier algebra becomes a Banach algebra. 
In fact, the multiplication of $ A(G) \cong L^1(\widehat{G}) $ identifies with the predual $ \widehat{\Delta}_* $ of the comultiplication $ \widehat{\Delta}: VN(G) \rightarrow VN(G) \overline{\otimes} VN(G)$.
Note also that $ A(G) $ is naturally an operator space via the operator space structure induced from $ L^1(\widehat{G}) $. 

\subsection{Beurling-Fourier algebras of locally compact quantum groups} \label{sec:Beurling-Fourier algebras of LCQG}

Let us first explain what we mean by a weight on the dual of a locally compact quantum group. The following definition is based on the one given in \cite{Lee_Beurling-Fourier}, compare also \cite{Ghandehari2021, Ludwig2012, Franz2021, Giselsson2024}.

\begin{defn} \label{defn:weight on the dual}
Let $ G $ be a locally compact quantum group. 
A densely defined unbounded operator $ w $ on $ L^2(G) $ is called a weight on the dual of $ G $ if the following conditions hold.
\begin{enumerate}[label=W\arabic*., series=W]
\item $ w $ is strictly positive self-adjoint 
\item $ w^{-1} \in VN(G) $
\item $ w^{-2} \otimes w^{-2} \leq \widehat{\Delta}(w^{-2}) $.
\end{enumerate}
A weight $ w $ is called symmetric if $ w^{-1} $ is invariant under the scaling group $ (\widehat{\tau}_t)_{t \in \mathbb{R}} $ and $ \widehat{R}(w^{-1}) = w^{-1} $. It is called central if $ w^{-1} $ is contained in the center of $ VN(G) $. 
\end{defn}

Note that condition $ W3 $ in Definition \ref{defn:weight on the dual} implies 
$
\|w^{-1} \|^{4} \leq \| \widehat{\Delta} (w^{-1} ) \|^{2} = \|w^{-1} \|^2
$
and hence $ \|w^{-1} \| \leq 1 $. 

\begin{rmk} \label{rmk:weight on the dual}
The choice of terminology in Definition \ref{defn:weight on the dual} is somewhat unfortunate because there appear other objects in our discussion 
that are also called weights, specifically the Haar weights on a locally compact quantum group and weights in the sense of Lie theory. 
However, this terminology is well-established in the context of Beurling-Fourier algebras, and we will not attempt to change it. In the sequel, it should always be clear from the context which objects are being referred to. 
\end{rmk}

Let $ M $ be a von Neumann algebra with predual $ M_* $. Then $ M_* $ has an $ M $-bimodule structure defined by
$$
(a \omega b, x) = (\omega, b x a)
$$
for $ \omega \in M_*, x \in M $ and $ a,b \in M $. 

\begin{prop} \label{prop:Beurling-Fourier algebra}
Let $ G $ be a locally compact quantum group and let $ w $ be a weight on the dual of $ G $. Then 
$$
L^1_w(\widehat{G}) = \{w^{-1} \omega \mid \omega \in L^1(\widehat{G}) \} \subseteq L^1(\widehat{G}) 
$$
is a subalgebra of $ L^1(\widehat{G}) $, which becomes a Banach algebra when endowed with the norm
$$
\|w^{-1} \omega \|_{L^1_w(\widehat{G})} = \| \omega \|_{L^1(\widehat{G})}. 
$$
\end{prop}

\begin{proof}
The proof is virtually the same as for classical locally compact groups, see \cite[Proposition~2.3]{Giselsson2024}. For the convenience of the reader we sketch the argument. 

Firstly, consider $ \Omega = \widehat{\Delta}(w)(w^{-1} \otimes w^{-1}) $. From $ w^{-2} \otimes w^{-2} \leq \widehat{\Delta}(w^{-2}) $ it follows that the a priori unbounded operator $ \Omega $ is in fact bounded with $ \|\Omega \| \leq 1 $.
For $ \omega, \eta \in L^1(\widehat{G}) $ and $ f \in L^\infty(\widehat{G}) $ we thus compute
\begin{align} \label{eq:Beurling-Fourier multiplication}
(\widehat{\Delta}_*(w^{-1} \omega \otimes w^{-1} \eta),f) &= (\omega \otimes \eta, \widehat{\Delta}(f) (w^{-1} \otimes w^{-1})) \nonumber \\
&= (\omega \otimes \eta, \widehat{\Delta}(f w^{-1})  \Omega) = (w^{-1} \widehat{\Delta}_*(\Omega (\omega \otimes \eta)), f). 
\end{align}
This shows in particular that $ L^1_w(\widehat{G}) $ is a subalgebra of $ L^1(\widehat{G}) $. Next, the map $ L^1(\widehat{G}) \rightarrow L^1_w(\widehat{G}), \omega \mapsto w^{-1} \omega $ is easily seen to be a bijection.
By definition of the norm on $ L^1_w(\widehat{G}) $, it follows that this map is in fact an isometric isomorphism. Using $ \|\Omega \| \leq 1 $ and \eqref{eq:Beurling-Fourier multiplication} one then checks that $ L^1_w(\widehat{G}) $ is a Banach algebra.
\end{proof}

Given a weight $ w $ on the dual of a locally compact quantum group $ G $, then since 
$ w^{-1} \omega_{\xi, \zeta} = \omega_{\xi, w^{-1} \zeta} $ for all $ \xi, \zeta \in L^2(G) $ we have 
\begin{equation} \label{eq:Beurling-Fourier algebra_alternative definition}
L^1_w(\widehat{G}) = \{\omega_{\xi, w^{-1} \zeta} \mid \xi , \zeta \in L^2(G) \} = \{ \omega_{\xi, \eta} \mid \xi \in L^2(G), \; \eta \in \Dom(w) \},
\end{equation}
where $ \Dom(w) \subseteq L^2(G) $ denotes the domain of the unbounded operator $ w $. This shows in particular that $ L^1_w(\widehat{G}) $ is dense in $ L^1(\widehat{G}) $. Also, if $ \xi \in L^2(G) $ and $ \eta \in \Dom(w) $, then by definition,
\begin{equation} \label{eq:Beurling-Fourier algebra norm bound}
\|\omega_{\xi,\eta} \|_{L^1_w(\widehat{G})} = \| \omega_{\xi, w \eta} \|_{A(G)} \leq \|\xi\| \|\eta \|_{\Dom(w)}
\end{equation}
where $\| \eta \|_{\Dom(w)} = \|\eta \| + \|w \eta \| $ is the graph norm on $ \Dom(w) $.

\begin{defn}
Let $ G $ be a locally compact quantum group and let  $ w $ be a weight on the dual of $ G $. Then 
$$ 
A(G,w) = \widehat{\lambda}(L^1_w(\widehat{G})) \subset A(G) 
$$
is called the Beurling-Fourier algebra associated with $ w $.
\end{defn}

The Beurling-Fourier algebra $ A(G,w) $ is naturally a Banach algebra and an operator space with the structures induced from $ L^1_w(\widehat{G}) $.

\subsection{Compact quantum groups} \label{sec:CQG}

In this section we gather some results from the theory of compact quantum groups. We mostly follow the notation of \cite[Section~4.2.3]{VoigtYuncken}, and refer to \cite{Klimyk} for more information.

A locally compact quantum group $ K $ is called compact if the Haar weight $ \varphi $ is a state on $ L^\infty(K) $.
In this case there exists a unique weak-$ * $ dense subalgebra $ \Cf^\infty(K) \subseteq L^\infty(K) $ 
of matrix coefficients, which can be endowed with a Hopf $ * $-algebra structure such that $ K $ becomes an  \emph{algebraic compact quantum group} in the sense of \cite{vanDaele}. 

Let $ \Irr(K) $ be the set of all equivalence classes of the irreducible unitary representations of $ K $. To each $ \mu \in \Irr(K) $ corresponds a unitary corepresentation 
matrix $ (u^\mu _{ij})_{1 \leq i,j \leq n_\mu} \in M_{n_{\mu}} (\Cf ^\infty (K)) $, where $ n_\mu \in \Nbb $ denotes the dimension of the underlying Hilbert space. 
Then $ (u^\mu _{ij})_{\mu \in \Irr(K), \; 1 \leq i, j \leq n_\mu} $ forms a vector space basis of $ \Cf^\infty(K) $. 
An arbitrary unitary representation of $ K $ can be written as a direct sum of irreducible representations, and the category of all finite dimensional unitary representations of $ K $ forms a rigid $ C^* $-tensor category. 
Given $ \mu \in \Irr(K) $ we write $ \overline{\mu} \in \Irr(K) $ for the conjugate representation, which is the unique irreducible unitary representation such that $ \mu \otimes \overline{\mu} $ contains the trivial representation. 

The comultiplication $ \Delta $, the counit $ \epsilon $ and the antipode $ S $ of $ \Cf^\infty(K) $ are characterised by
\begin{equation*} \label{eq:comultiplication, counit, and antipode of CQG}
\Delta(u^\mu _{ij}) = \sum_{1 \leq k \leq n_\mu} u^\mu _{ik} \otimes u^\mu _{kj} , \quad \epsilon(u^\mu_{ij}) = \delta_{ij}, \quad S(u^\mu _{ij}) = (u^\mu_{ji})^*
\end{equation*}
for all $\mu \in \Irr(K), \; 1 \leq i,j \leq n_\mu $.

Let us denote the algebraic quantum group dual to $ \Cf^\infty (K)$ by $ \Dcal(K) $. Then there exists a family of elements $ (\omega^\mu _{ij})_{\mu \in \Irr(K), \; 1 \leq i, j \leq n_\mu} \subseteq \Dcal (K)$ defined by
\begin{equation*} \label{eq:basis of discrete quantum group}
(\omega^\lambda _{ij}, u^\mu _{kl}) = \delta_{\mu \nu} \delta_{ik} \delta_{jl} 
\end{equation*}
for all $\lambda, \mu \in \Irr(K), \; 1 \leq i,j \leq n_\lambda, \; 1 \leq k, l \leq n_\mu $. These elements form a basis of $ \Dcal(K) $ and are in fact matrix units for this $ * $-algebra. That is, for all $ \lambda, \mu \in \Irr(K) $ and $ 1 \leq i,j \leq n_\lambda $ and $ 1 \leq k, l \leq n_\mu $, we have 
\begin{equation*} \label{eq:matrix elements of discrete quantum group}
\omega^\lambda _{ij} \omega^\mu _{kl} = \delta_{\lambda \mu} \delta_{jk} \omega^\lambda _{il}, \quad (\omega^\mu _{ij})^* = \omega^\mu _{ji}.
\end{equation*}
Therefore, as a $ * $-algebra,
\begin{equation} \label{eq:dual discrete quantum group}
\Dcal(K) = \bigoplus_{\mu \in \Irr(K)} M_{n_\mu}(\Cbb)
\end{equation}
is an algebraic direct sum of full matrix algebras. 

For each $ \mu \in \Irr(K) $ we have an irreducible $ * $-representation $ \pi_\mu: \Dcal(K) \rightarrow M_{n_\mu}(\Cbb) $ given by the projection onto the $ \mu $-component of this direct sum, and all irreducible $ * $-representations of $ \Dcal(K) $ arise in this way. If we write $ V(\mu) $ for the representation space of $ \pi_\mu $ then the decomposition in \eqref{eq:dual discrete quantum group}can be written as
\begin{equation*} \label{eq:decomposition of discrete algebraic quantum group}
\Dcal(K) = \bigoplus_{\mu \in \Irr(K)} \End( V(\mu)).
\end{equation*}
We shall also write
\begin{equation} \label{eq:basis of compact quantum group}
u^\mu _{ij} = \la e_i ^\mu | \cdot | e_j ^\mu \ra \in \End (V(\mu))^*
\end{equation}
where $ (e_i^\mu)_{1 \leq i \leq n_\mu} $ is a fixed orthonormal basis of the Hilbert space $ V(\mu) $ and $ \la v | \cdot | w \ra = \omega_{v,w} \in \Lbb (V(\mu))_* = \End(V(\mu))^* $. Thus,
\begin{equation*} \label{eq:decomposition of compact algebraic quantum group}
\Cf^\infty (K) = \bigoplus_{\mu \in \Irr(K)} \End(V(\mu))^*  
\end{equation*}
as a coalgebra.

Denote the GNS-map of the Haar state $ \varphi $ by $\Lambda: \Cf^\infty(K) \rightarrow L^2(K) $. 
Then $ \Cf^\infty(K) $ acts on $ L^2(K) $  by left translations,  
\begin{equation*} \label{eq:L2 action of compact quantum group}
f \Lambda(g) = \Lambda(fg), \quad f,g \in \Cf^\infty(K), 
\end{equation*}
and $ \Dcal(K) $ acts on $ L^2(K) $ by convolution, 
\begin{equation} \label{eq:L2 action of discrete quantum group}
\omega \Lambda(g) = \Lambda \Big( \big( (\omega S^{-1}) \otimes \id \big) \Delta(g) \Big), \quad \omega \in \Dcal(K), \; g \in \Cf^\infty(K).
\end{equation}
Moreover, $ L^\infty(K) $ and $ VN(K) $ are the weak-$ * $ closures of $ \Cf^\infty(K) $ and $ \Dcal(K) $ in $ \Lbb(L^2(K)) $ via these representations. 
In particular, 
\begin{align}\label{eq:decopmposition of discrete quantum group}
VN(K) &= l^\infty\text{-}\prod_{\mu \in \Irr(K)} \End (V(\mu)) \nonumber \\
&= \{X = (X_\mu) \in \prod_{\mu \in \Irr(K)} \End (V(\mu)) \mid \sup_{\mu \in \Irr(K)} \| X_\mu \| < \infty \}.
\end{align}
Using the notation introduced further above, the Kac-Takesaki operator $ \ww \in \Lbb(L^2(K) \otimes L^2(K)) $ can be written as
\begin{equation} \label{eq:Fundamental multiplicative unitary}
\ww = \sum_{\mu \in \Irr(K)} \sum_{1 \leq i,j \leq n_\mu} u^\mu _{ij} \otimes \omega^\mu _{ij}, 
\end{equation}
with convergence understood in the strong operator topology. 
Note that $ \Cf^\infty (K) $ is naturally contained in $ A(K) $. Explicitly, we have
\begin{equation} \label{eq:algebraic compact quantum group embeds into the Fourier algebra}
\Cf^\infty(K) \ni \sum_{i} (\id \otimes \varphi) \big(\Delta(f_i)^* (1 \otimes g_i) \big) \longmapsto \hat{\lambda} \Big( \sum_{i} \omega_{f_i ,g_i} \Big) \in A(K)
\end{equation}
for all $ f_i, g_i \in \Cf^\infty(K) $.

\subsection{Weights on duals of compact quantum groups} \label{sec:weight on the dual of CQG}

If $ K $ is a compact quantum group then we have 
$$
VN(K) \otimes VN(K) \subseteq \prod_{\lambda, \mu \in \Irr(K)} \End(V(\lambda)) \otimes \End(V(\mu)).
$$
Moreover, since every $ * $-representation of $ \Dcal(K) $ is completely reducible we get, for each $ \lambda, \mu \in \Irr(K) $,
$$
\pi_\lambda \otimes \pi_ \mu \cong \bigoplus_{j = 1}^k \pi_{\nu_j}
$$
for some $ \nu_1, \dots, \nu_k \in \Irr(K) $. We express this decomposition as
$$
\pi_\lambda \otimes \pi_\mu \cong \bigoplus_{\nu \subseteq \lambda \otimes \mu} \pi_\nu.
$$
Therefore, for each $ \lambda, \mu \in \Irr(K) $, there exists a unitary $ U^{\lambda \mu}: V(\lambda) \otimes V(\mu) \rightarrow \bigoplus_{\nu \subseteq \lambda \otimes \mu} V(\nu) $ 
such that
$$ 
C_{U^{\lambda \mu}} (\pi_\lambda \otimes \pi_{\mu}) = U^{\lambda \mu} (\pi_\lambda \otimes \pi_{\mu}) (U^{\lambda \mu})^* = \bigoplus_{\nu \subseteq \lambda \otimes \mu} \pi_\nu, 
$$
where $ C_{U^{\lambda \mu}} $ is the conjugation by $ U^{\lambda \mu} $. Defining $ U = \prod_{\lambda, \mu \in \Irr(K)} U^{\lambda \mu} $ and $ C_U = \prod_{\lambda, \mu \in \Irr(K)} C_{U^{\lambda \mu}} $, we get an isomorphism of $ * $-algebras
\begin{equation*} \label{eq:isomorphism of tensor decomposition}
C_U: \prod_{\lambda, \mu \in \Irr(K) } \End (V(\lambda)) \otimes \End (V(\mu)) \longrightarrow 
\prod_{\lambda, \mu \in \Irr(K) } \Big( \bigoplus_{\nu \subseteq \lambda \otimes \mu} \End(V(\nu)) \Big).
\end{equation*}
We then have the following fact, see \cite[Proposition~3.1]{Franz2021}.

\begin{lem} \label{lem:VN(K) comultiplication}
For $ X = (X_\mu)_{\mu \in \Irr(K)} \in VN(K) $ we have
\begin{equation*} \label{eq:VN(K) comultiplication}
C_U(\widehat{\Delta} (X) ) = ( \bigoplus_{\nu \subseteq \lambda \otimes \mu} X_\nu )_{\lambda , \mu} \in \prod_{\lambda, \mu \in \Irr (K)} \Big( \bigoplus_{\nu \subseteq \lambda \otimes \mu } \End (V(\nu)) \Big).
\end{equation*}
\end{lem}

\begin{proof}
For the convenience of the reader we shall give a proof.

For each $ \lambda \in \Irr(K) $, let $ (e^\lambda _i)_{1 \leq i \leq n_\lambda} $ be the orthonormal basis as in \eqref{eq:basis of compact quantum group}. Now, observe that 
for $ \lambda, \mu, \eta \in \Irr(K) $ and indices corresponding to each irreducible representation involved,
\begin{align*}
(\widehat{\Delta} &( \omega^\eta _{ij}), u^{\lambda} _{kl} \otimes u^{\mu} _{mn}) = (\omega^\eta _{ij} , u^\mu _{mn} u^\lambda _{kl} ) 
= \Big(\omega^\eta _{ij}, \la e_k ^\lambda \otimes e_m ^\mu | \cdot | e_l ^\lambda \otimes e_n ^\mu \ra \Big) \\
&= \la e_k ^\lambda \otimes e_m ^\mu | (\pi_\lambda \otimes \pi_\mu) (\omega^\eta _{ij} ) | e_l ^\lambda \otimes e_n ^\mu \ra \\
&= \Big\la U^{\lambda \mu} (e_k ^\lambda \otimes e_m ^\mu) \Big| \bigoplus_{\nu \subseteq \lambda \otimes \mu} \pi_{\nu} (\omega^\eta _{ij} ) \Big| U^{\lambda \mu} (e_l^\lambda \otimes e_n^\mu) \Big\ra \\
&= \Big\la U \big((\delta_{\theta \lambda} \delta_{\mu \zeta} e_k^\lambda \otimes e_m^\mu)_{\theta \zeta} \big)\Big| \big( \bigoplus_{\nu \subseteq \theta \otimes \zeta} \pi_{\nu} (\omega^\eta_{ij} ) \big)_{\theta \zeta} \Big| U \big((\delta_{\theta \lambda} \delta_{\mu \zeta} e_l^\lambda \otimes e_n^\mu)_{\theta \zeta})\Big\ra \\
&= \Big\la (\delta_{\theta \lambda} \delta_{\mu \zeta} e_k^\lambda \otimes e_m^\mu)_{\theta \zeta} \Big| 
C_U^{-1} \Big(\big(\bigoplus_{\nu \subseteq \theta \otimes \zeta} \pi_{\nu} (\omega^\eta_{ij}) \big)_{\theta \zeta} \Big) \Big| 
(\delta_{\theta \lambda} \delta_{\mu \zeta} e_l^\lambda \otimes e_n^\mu)_{\theta \zeta} \Big\ra \\
&= \Big(C_U ^{-1} \Big( \big( \bigoplus_{\nu \subseteq \theta \otimes \zeta} \delta_{\eta \nu} \omega^{\eta} _{ij} \big)_{\theta \zeta} \Big), u^{\lambda} _{kl} \otimes u^{\mu} _{mn} \Big).
\end{align*}
So, for $ \eta \in \Irr(K) $ and $ 1 \leq i,j \leq n_\eta $, we get
$$
C_U\big(\widehat{\Delta}(\omega^\eta _{ij})\big)_{\lambda, \mu}
= \bigoplus_{\nu \subseteq \lambda \otimes \mu} \delta_{\nu\eta} \omega^{\eta}_{ij}.
$$
This yields the claim since $ \Dcal(K) \subset VN(K) $ is weak-$ * $ dense.
\end{proof}

For a central weight $ w $ on the dual of a compact quantum group $ K $ we shall write $ w = (w(\mu))_{\mu \in \Irr(K)} $ in the decomposition \eqref{eq:decopmposition of discrete quantum group} in the sequel. That is, we will identify $ w $ with a function $ w: \Irr(K) \rightarrow (0,\infty) $. 

\begin{prop} \label{prop:central weight on the dual}
Let $ K $ be a compact quantum group and let $ w: \Irr(K) \rightarrow (0,\infty) $. Then $ w $ is a central weight on the dual of $ K $ if and only if the following conditions are satisfied:
\begin{enumerate}[label=Z\arabic*., series=Z]
\item $ w(\mu) \geq 1 $ for all $ \mu \in \Irr(K) $
\item $ w(\nu) \leq w(\lambda) w(\mu) $ for all $ \lambda, \mu, \nu \in \Irr(K) $ such that $ \nu \subseteq \lambda \otimes \mu $.
\end{enumerate}
Moreover $ w $ is symmetric if and only if $ w(\mu) = w(\overline{\mu}) $ for all $ \mu \in \Irr(K) $.
\end{prop}

\begin{proof}
Let $ w = (w(\mu))_{\mu \in \Irr(K)} $ be a central weight on the dual of $ K $. Then $ 0 < w^{-1}(\mu) \leq 1 $ for all $ \mu \in \Irr(K) $ by condition W1 and 
the fact that $ \|w^{-1} \| \leq 1 $, so that $ w(\mu) \geq 1 $ for all $ \mu \in \Irr(K) $. 
Thanks to Lemma~\ref{lem:VN(K) comultiplication}, condition W3 translates to
$$
w^{-2} (\lambda) w^{-2} (\mu) \leq w^{-2} (\nu)
$$
for all $ \lambda, \mu \in \Irr(K) $ and $ \nu \subseteq \lambda \otimes \mu $, which is equivalent to Z2. 

Conversely, given a function $ w: \Irr(K) \rightarrow (0,\infty) $ satisfying Z1--Z2, the unbounded operator $ w $ on $ L^2(K) $ such that $ w^{-1} = (w(\mu)^{-1}) \in VN(K) $ is easily checked to be a weight on the dual of $ K $.

Since the (unitary) antipode exchanges 
the matrix blocks corresponding to $ \mu $ and $ \overline{\mu} $ one finds that $ w $ is symmetric if and only if $ w(\mu) = w(\overline{\mu}) $ for all $ \mu \in \Irr(K) $.
\end{proof}

\begin{rmk}
Let $ w $ be a symmetric central weight on the dual of a compact quantum group $ K $ in the sense of Definition \ref{defn:weight on the dual}. If the value $ w(\epsilon) $ of $ w $ on the trivial representation $ \epsilon $ equals $ 1 $, then $ L = \log(w) $ is a length function on $ \widehat{K} $ in the sense of \cite{Vergnioux2007}. 
Note that 
we can always stipulate $ w(\epsilon) = 1 $ without affecting the conditions in Proposition \ref{prop:central weight on the dual}. That is, symmetric central weights on duals of compact quantum groups are essentially the same thing as length functions. 
\end{rmk}

Throughout the rest of the paper, we will be mostly interested in symmetric central weights on duals of compact quantum groups. We note that some of the arguments below work in the same way without the symmetry condition.

\subsection{Beurling-Fourier algebras of compact quantum groups} \label{Beurling-Fourier algebras of CQG}

Let $ K $ be a compact quantum group and let $ w $ be a symmetric central weight on the dual of $ K $. 
Recall that $ \Cf^\infty(K) $ embeds into $ A(K) $ via \eqref{eq:algebraic compact quantum group embeds into the Fourier algebra}. Since $ \Lambda(\Cf^\infty(K)) \subseteq L^2(K) $ is densely contained in $ \Dom(w) $ with respect to the graph norm, it follows from \eqref{eq:Beurling-Fourier algebra_alternative definition} and \eqref{eq:Beurling-Fourier algebra norm bound} that we have the following dense inclusion of algebras
\begin{equation} \label{eq:Cinfty(K) embeds into A(K,W)}
\Cf^\infty(K) \ni \sum_i (\id \otimes \varphi) \big(\Delta(f_i)^* (1 \otimes g_i^*) \big) \longmapsto \hat{\lambda} \Big( \sum_i \omega_{f_i, g_i} \Big) \in A(K,w).
\end{equation}
We shall be interested in conditions which guarantee that an 
algebra representation $ \theta: \Cf^\infty(K) \rightarrow \Lbb(\Hcal) $ extends to a completely bounded (CB) representation $ \tilde{\theta}: A(K,w) \rightarrow \Lbb(\Hcal) $. Recalling the expression of the fundamental multiplicative unitary 
in \eqref{eq:Fundamental multiplicative unitary}, we define
\begin{equation} \label{eq:definition of v_varphi}
v_\theta = (\theta \otimes \id)(\ww) = \Big(\sum_{1 \leq i,j \leq n_\mu} \theta(u^\mu_{ij}) \otimes \omega^\mu _{ij} \Big)_{\mu} 
\in \prod_{\mu \in \Irr(K)} \Lbb(\Hcal) \otimes \Lbb(V(\mu)).
\end{equation}
We denote the $ \mu $-component of $ v_\theta $ by $ v_\theta^\mu $.
Since $ {l^\infty}\text{-}\prod_{\mu \in \Irr(K)} \Lbb(\Hcal) \overline{\otimes} \Lbb(V(\mu)) \cong \Lbb(\Hcal) \overline{\otimes} VN(K) $ we have
\begin{equation} \label{eq:supremum condition on v_varphi}
v_\theta \in \Lbb(\Hcal) \otimes VN(K) \Longleftrightarrow \sup_{\mu \in \Irr(K)} \| v_\theta^\mu \| < \infty.
\end{equation}
Next recall that if $ M, N $ are von Neumann algebras and $ CB(M_*, N) $ denotes the space of all CB linear maps from $ M_* $ into $ N $ then we have a canonical isomorphism
$$
N \overline{\otimes} M \cong CB(M_*, N),
$$
induced by the correspondence
\begin{equation} \label{eq:the set of CB maps characterization}
N \overline{\otimes} M \ni n \otimes m \longmapsto (m, \cdot \;) n \in CB(M_*, N), 
\end{equation}
see \cite[Theorem 2.5.2]{Pisier}. Observe that, for $ f,g \in \Cf^\infty(K) $,
\begin{align*}
(\id \otimes \omega_{f, g})(v_\theta) &= \sum_{\mu \in \Irr(K)} \sum_{1 \leq i, j \leq n_\mu} \theta(u^\mu _{ij}) \langle \Lambda(f), \omega^\mu _{ij} \Lambda(g) \rangle\\
&= \sum_{\mu \in \Irr(K)} \sum_{1 \leq i, j \leq n_\mu} \theta(u^\mu _{ij}) \big(\omega^\mu_{ij}, S^{-1} (\id \otimes \varphi) \big((1 \otimes f^*) \Delta(g) \big) \big) \\
&= \sum_{\mu \in \Irr(K)} \sum_{1 \leq i, j \leq n_\mu} \theta(u^\mu _{ij}) \big(\omega^\mu_{ij}, (\id \otimes \varphi) \big( \Delta(f)^* (1 \otimes g ) \big) \big) \\
&= \theta \Big((\id \otimes \varphi) \big(\Delta(f)^* (1 \otimes g) \big) \Big), 
\end{align*}
where we use \eqref{eq:L2 action of discrete quantum group} in the second equality and left invariance of the Haar state $ \varphi $ in the third. Note also that all sums in this calculation are finite. 
In combination with \eqref{eq:algebraic compact quantum group embeds into the Fourier algebra}, it follows that if $ \theta $ extends to a CB map $ \tilde{\theta}: A(K) \rightarrow \Lbb(\Hcal) $ then $ v_\theta $ is the element in $ \Lbb(\Hcal) \overline{\otimes} VN(K) $ corresponding to $ \theta $ under the identification
$$
\Lbb(\Hcal) \overline{\otimes} VN(K) \cong CB(A(K), \Lbb(\Hcal))
$$
given via \eqref{eq:the set of CB maps characterization}. In fact, the condition $ v_\theta \in \Lbb(\Hcal) \overline{\otimes} VN(K) $ is equivalent to the CB-extendability of $ \theta $ to $ A(K) $.

Recall from the proof of Proposition \ref{prop:Beurling-Fourier algebra} that the map
$$
A(K) \ni \widehat{\lambda}(\omega) \longmapsto \widehat{\lambda}(w^{-1} \omega) \in A(K,w)
$$
is an isometric isomorphism. In particular, we can identify $ A(K,w) \cong VN(K)_* $ via the pairing
$$ 
(\;\cdot\;, \;\cdot\; )_w: A(K,w) \otimes VN(K) \ni (\widehat{\lambda}(w^{-1} \omega), x) \longmapsto (\omega, x) \in \Cbb.
$$
Using this identification, we can apply the same reasoning as above to $ A(K,w) $ and arrive at the following result.  

\begin{prop} \label{prop:CB extension condition}
Let $ K $ be a compact quantum group. 
Given a symmetric central weight $ w = (w(\mu))_{\mu \in \Irr(K)} $ on the dual of $ K $, an algebra representation $ \theta: \Cf^\infty(K) \rightarrow \Lbb(\Hcal) $ extends to a completely bounded algebra 
representation $ \tilde{\theta}: A(K,w) \rightarrow \Lbb(\Hcal) $ if and only if $ v_\theta(1 \otimes w^{-1}) \in \Lbb(\Hcal) \overline{\otimes} VN(K) $, which in turn is equivalent to the condition
\begin{equation*} \label{eq:CB extension condition}
\sup_{\mu \in \Irr(K)} \frac{\| v_\theta^\mu \|}{w(\mu)} < \infty.
\end{equation*}
\end{prop}

\begin{proof}
By the above calculation we obtain, by slight abuse of notation,  
\begin{align*}
\big((\id \otimes \omega_{f,g}), v_\theta (1 \otimes w^{-1}) \big)_w = \big((\id \otimes w^{-1} \omega_{f,g} ), v_\theta)_w 
= ((\id \otimes \omega_{f,g}), v_\theta) \\ = \theta \Big((\id \otimes \varphi) \big(\Delta(f)^* (1\otimes g) \big) \Big) 
\end{align*}
for any $ f, g \in \Cf^\infty(K) $. 
Hence the claim follows from \eqref{eq:supremum condition on v_varphi}.
\end{proof}

\section{Previous works} \label{sec:previous works}

In this section, we briefly summarize some of the results of \cite{Ludwig2012} and \cite{Franz2021} which have relevance to this paper.

\subsection{The work of Ludwig, Spronk, and Turowska} \label{sec:the work of Ludwig et al.}

In \cite{Ludwig2012}, Ludwig, Spronk and Turowska considered central weights on the dual of a classical compact group $ K $. They found that for any such weight $ w $ one has dense inclusions of algebras
$$
\Cf^\infty(K) \subseteq A(K, w) \subseteq A(K).
$$
This implies
\begin{equation} \label{eq:inclusions of the spectra}
\Spec(A(K)) \subseteq \Spec(A(K,w)) \subseteq \Spec(\Cf^\infty(K)), 
\end{equation}
where $ \Spec(A) = \{\chi: A \rightarrow \Cbb \mid \chi \text{ is a nonzero algebra homomorphism} \} $ for an algebra $ A $, and the inclusions are provided via the 
obvious restriction maps.

By Eymard's theorem \cite[Th\'eor\`eme 3.34]{Eymard1964} we have $ \Spec(A(K)) \cong K $, and if $ K $ is a compact connected Lie group then $ \Spec(\Cf^\infty(K)) \cong K_\Cbb $ is the \textit{Chevalley complexification of $ K $}, 
compare \cite[Chapter VI]{Chevalley}. These identifications are provided by ``evaluation at the points of $ K $ and $ K_\Cbb $", respectively.
It follows that $ \Spec(A(K,w)) $ is placed between $ K $ and $ K_\Cbb $. 

In addition, the spectra $ \Spec(A(K,w)) $ exhaust $ K_\mathbb{C} $ as $ w $ ranges over all weights. More precisely, we have the following result, compare \cite[Section 5.1]{Ghandehari2021} for the case $ K = SU(N) $. 

\begin{prop} \label{prop:the main result 2 of Ludwig et al.}
Let $ K $ be a compact connected Lie group. Then 
\begin{equation} \label{eq:the main result 2 of Ludwig et al.}
K_\Cbb = \bigcup_{w} \Spec(A(K,w)),
\end{equation}
where the union is taken over all central weights on the dual of $ K $. 
\end{prop}

\begin{proof}
Firstly, introduce an $ \Ad $-invariant inner product $ (\cdot, \cdot): \kf \times \kf \rightarrow \Rbb $ and let $ \{X_1, \dots, X_d \} \subseteq \kf $ be an orthonormal basis of $ \kf $ with respect to this inner product. Each $ \pi \in \Irr(K) $ induces a Lie algebra representation
$$
\kf \ni X \longmapsto \left. \frac{d}{dt} \right|_{t=0} \pi (\exp (tX)) \in \End(V(\pi)),
$$
which, by slight abuse of notation, we also denote by $ \pi: \kf \rightarrow \End(V(\pi)) $. 
The Casimir operator
$$
\pi(X_1)^2 + \cdots + \pi(X_d)^2 \in \End(V(\pi))
$$
does not depend on the choice of the orthonormal basis and is always a negative scalar matrix. If we write 
$$
\pi(X_1)^2 + \cdots + \pi(X_d)^2 = - c(\pi) \id_{V(\pi)},
$$
then for all $ \lambda, \mu, \nu \in \Irr(K) $ with $ \nu \subseteq \lambda \otimes \mu $ we have  
$$
c(\nu)^{\frac{1}{2}} \leq c(\lambda)^{\frac{1}{2}} + c(\mu)^{\frac{1}{2}}, 
$$
according to \cite[Lemma~5.3]{Ludwig2012}. As a consequence, Proposition~\ref{prop:central weight on the dual} shows that, for each $ \beta \geq 0 $, the function $ w_\beta^{LST}: \Irr(K) \rightarrow (0, \infty) $ given by
$$
w_\beta^{LST}(\mu) = e^{\beta c(\mu)^{\frac{1}{2}}}
$$
defines a central weight $ w_\beta^{LST} = (w_\beta^{LST} (\mu))_{\mu \in \Irr(K)} $ on the dual of $ K $.

Let $ H \leq K $ be a closed abelian subgroup with Lie algebra $ \hf \leq \kf $. Since the values $ c(\mu) $ for $\mu \in \Irr(K) $ do not depend on the choice of the orthonormal basis we may assume that $ \{X_1, \dots, X_k \} $ is a basis of $ \hf $. For each $ \beta > 0 $, 
$$
w_\beta^{H} = \Big( e^{\beta (|\pi(X_1)| + \cdots + |\pi(X_k)|) } \Big)_{\pi \in \Irr(K)} \in \prod_{\pi \in \Irr(K)} \End(V(\pi))
$$
is a weight on the dual of $ K $ of the form considered in \cite[Eq.~(2.8)]{Lee_Beurling-Fourier}. If we regard $ \Spec(A(K, w_\beta^{H})) $ as a subset of $ \Spec(\Cf^\infty(K)) \cong K_\Cbb $ via \eqref{eq:inclusions of the spectra}, it then follows from  \cite[Theorem~2.11]{Lee_Beurling-Fourier} that we have
\begin{equation} \label{eq:calculation of spectra-LL}
\Spec(A(K, w_\beta^{H})) = \{ s \exp(iX) \mid s \in K, \; X \in \hf, \; |X|_{\infty} \leq \beta \},
\end{equation}
where $ |X|_{\infty} = \max_{1 \leq j \leq k} |x_j| $ for $ X = x_1 X_1 + \cdots + x_k X_k \in \hf $.

Let $ \pi \in \Irr(K) $. Since $ \{\pi(X_1), \dots, \pi(X_d)\} \subseteq \End(V(\pi)) $ are skew-adjoint operators and $ \pi(X_i) $ and $ \pi(X_j) $ mutually commute for $ 1 \leq i,j \leq k $, we have
$$
|\pi(X_1)| + \cdots + | \pi(X_k) | \leq \sqrt{k} \; \big( - \pi(X_1)^2 - \cdots - \pi(X_k)^2 \big)^{\frac{1}{2}}
\leq \sqrt{k} \; c(\pi)^{\frac{1}{2}} \id_{V(\pi)}.
$$
Hence
$$
\id_{V(\pi)} \leq e^{ \sqrt{k} \; c(\pi)^{\frac{1}{2}} \id_{V(\pi)} - ( |\pi (X_1)| + \cdots + | \pi (X_k)| )} = e^{ \sqrt{k} \; c(\pi)^{\frac{1}{2}}} e^{ - (| \pi (X_1) | + \cdots + | \pi (X_k) | ) }
$$
in $ \End(V(\pi)) $. This implies that, for all $ \beta \geq 0 $,
$$
(w_{\sqrt{k}\beta}^{LST}) ^{-1} = \Big( e^{-\beta \sqrt{k}\; c(\pi)^{\frac{1}{2}} } \Big)_{\pi \in \Irr(K)} \leq \Big(e^{-\beta ( |\pi(X_1)| + \cdots + | \pi(X_k) | ) } \Big)_{\pi \in \Irr(K)} = (w_\beta^{H})^{-1}
$$ 
as elements of $ VN(K) $. Since $ w_\beta^{LST} $ is central, these two weights strongly commute, which enables us to apply \cite[Proposition~3.7]{Lee_Beurling-Fourier} to conclude
$$
\Spec(A(K, w_{\sqrt{k}\beta}^{LST})) \supseteq \Spec(A(K, w_\beta^{H})).
$$
Together with \eqref{eq:calculation of spectra-LL} this implies
$$
K_\Cbb \supseteq \bigcup_{\beta \geq 0} \Spec(A(K,w_\beta^{LST})) \supseteq \{s \exp(iX) \mid s \in K, \; X \in \kf \} = K_\Cbb,
$$
keeping in mind that $ H \leq K $ was an arbitrary closed abelian subgroup of $ K $. This completes the proof.
\end{proof}

As a consequence, the Beurling-Fourier algebras $ A(K,w) $ can be used to ``detect" the complexification of the compact Lie group  $ K $. 
In this setting, the weight $ w $ plays the role of a parameter specifying the ``fattening" of the original group $ K $ inside its complexification $ K_\Cbb $.

\subsection{The work of Franz and Lee} \label{sec:the work of Franz and Lee}

In the later work \cite{Franz2021}, Franz and Lee generalised these considerations to the realm of compact quantum groups. The dense inclusions of algebras
$$
\Cf^\infty(K) \subseteq A(K,w) \subseteq A(K)
$$
still hold in this case by \eqref{eq:Cinfty(K) embeds into A(K,W)}. However, as these algebras are noncommutative in the quantum case, the associated inclusions
$$
\Spec(A(K)) \subseteq \Spec(A(K,w)) \subseteq \Spec(\Cf^\infty(K))
$$
contain only relatively little information about the ``complexification" of $ K $.

It is shown in \cite{Franz2021} that $ \Spec(A(K)) $ is homeomorphic to the maximal classical subgroup $ \tilde{K} $ of $ K $, and $ \Spec(\Cf^\infty(K)) $ is the ``complexification" of $ \tilde{K} $ in the sense of Mckennon \cite{McKennon1979}. Thus, the best that $ \Spec(A(K,w)) $ can do is to give information about the complexification 
of the classical subgroup $ \tilde{K} $ of $ K $. As already pointed out above, $ \tilde{K} $ is typically relatively small. For instance, if $ K $ is the $ q $-deformation of a compact semisimple group, then $ \tilde{K} $ identifies with the maximal torus of $ K $. Still, it is shown 
in \cite{Franz2021} that $ \Spec(A(K,w)) $ can be used to detect the complexification of $ \tilde{K} $ for the free orthogonal and unitary quantum groups $ K = O^+ _F, U^+ _F $, and the quantum permutation group $ S_n ^+ $.

Franz and Lee pushed this program further in the case $ K = SU_q(2) $, and considered more general representations of $ \Cf^\infty(K) $ instead of only the one-dimensional representations that constitute $ \Spec(\Cf^\infty(K)) $.
Let us review some of the ingredients in this analysis, referring to subsection \ref{sec:the q-deformation of compact semisimple Lie groups} below for more details. 

Fix $ -1 < q < 1 $ such that $ q \neq 0 $. 
Let $ \Cf_c^\infty(SL_q(2, \Cbb)) $ be the algebraic quantum group underlying the quantum Lorentz group $ SL_q (2, \Cbb) $ 
introduced in \cite{Podles1990}, 
given by 
$$
\Cf_c^\infty(SL_q (2, \Cbb)) \cong \Cf^\infty(SU_q(2)) \otimes \Dcal(SU_q(2)), 
$$
where $ \Dcal(SU_q(2)) $ is the algebraic quantum group dual to $ \Cf^\infty(SU_q(2)) $, see \cite{VoigtYuncken}. 
If $ \SSp(A) $ denotes the set of equivalence classes of irreducible $ * $-representations 
of a $ * $-algebra $ A $ on Hilbert spaces then we have
\begin{equation*} \label{eq:sp Ccinfty(SLq(2,C))}
\SSp(\Cf_c^\infty (SL_q (2, \Cbb))) \cong \SSp(\Cf^\infty(SU_q(2))) \times \SSp(\Dcal(SU_q(2))),
\end{equation*}
given by the following correspondence
$$ 
(\pi_c, \pi_s) \longmapsto \pi_c \otimes \pi_s. 
$$
Here, $ \pi_s \in \SSp(\Dcal(SU_q(2))) $ is the $ * $-representation corresponding to $ s \in \frac{1}{2}\Nbb $, see \cite[Section~3.5 and Section~3.10]{VoigtYuncken}. Denote its representation space by $V(s)$.

Also, $ \Cf^\infty(SU_q(2)) $ admits an algebra embedding
$$
\Cf^\infty(SU_q(2)) \ni f \longmapsto i (f \bowtie 1) \in \Cf^\infty(SL_q(2, \Cbb)),
$$
where $ \Cf^\infty(SL_q(2, \Cbb)) = \prod_{s \in \frac{1}{2}\Nbb} \Cf^\infty(SU(2)) \otimes \End(V(s)) $ is the algebraic multiplier algebra of $ \Cf_c^\infty (SL_q (2, \Cbb)) $ 
and $ i: \Cf^\infty(SU_q(2)) \bowtie \Cf^\infty(SU_q(2)) \rightarrow \Cf^\infty(SL_q(2, \Cbb)) $ 
is the map of \cite[Proposition~4.24]{VoigtYuncken}. Let us point out that this embedding of $ \Cf^\infty(SU_q(2)) $ is not compatible with the $ * $-structures, and that it should be viewed as embedding holomorphic polynomial functions on $ SL_q(2, \Cbb) $ into the algebra of all smooth functions. 

Note that every $ \pi \in \SSp(\Cf_c^\infty(SL_q(2, \Cbb)) $ factorises through $ \Cf^\infty(SU_q(2)) \otimes \Lbb(V(s)) $ for some $ s \in \frac{1}{2} \Nbb $. 
It follows that $ \pi: \Cf^\infty_c(SL_q(2, \Cbb)) \rightarrow \Lbb(\Hcal_\pi) $ admits an extension $ \tilde{\pi}: \Cf^\infty(SL_q(2, \Cbb)) \rightarrow \Lbb(\Hcal_\pi) $, 
and we get an algebra homomorphism
\begin{equation*} \label{eq:representation of Cinfty(SUq(2))}
\theta_\pi: \Cf^\infty(SU_q(2)) \xrightarrow{i} \Cf^\infty(SL_q(2, \Cbb)) \xrightarrow{\tilde{\pi}} \Lbb(\Hcal_\pi).
\end{equation*}
Now, fix $ \beta \geq 1 $ and let $ w^{FL}_\beta = (w^{FL}_\beta(s))_{s \in \frac{1}{2} \Nbb} $ be the weight on the dual of $ SU_q(2) $ defined by
\begin{equation} \label{eq:weight on the dual of SUq(2)}
w^{FL}_\beta(s) = \beta^{2s}, \quad s \in \frac{1}{2} \Nbb.
\end{equation}
Then one can ask which representations of the form $ \theta_\pi: \Cf^\infty(SU_q(2)) \rightarrow \Lbb(\Hcal_\pi) $ extend to CB algebra 
homomorphisms $ \tilde{\theta}_\pi: A(SU_q(2), w^{FL}_\beta) \rightarrow \Lbb(\Hcal_\pi) $, and Franz and Lee found the following criterion \cite[Theorem 5.6]{Franz2021}. 

\begin{thm} \label{thm:the main result of Franz and Lee}
Let $ \pi = \pi_c \otimes \pi_s \in \SSp(\Cf^\infty (SL_q(2, \Cbb))) $ with $ s \in \frac{1}{2}\Nbb $. Then $ \theta_\pi: \Cf^\infty(SU_q(2)) \rightarrow \Lbb(\Hcal_\pi) $ 
extends to a completely bounded algebra homomorphism $ A(SU_q(2), w^{FL}_\beta) \rightarrow \Lbb(\Hcal_\pi) $ if and only if
\begin{equation*} \label{eq:the main result 2 of Franz and Lee-estimate}
|q|^{-s} \leq \beta.
\end{equation*}
Therefore,
\begin{align*}\label{eq:the main result 2 of Franz and Lee-spectrum}
&\SSp(\Cf_c^\infty (SL_q(2, \Cbb))) \\
&= \bigcup_{\beta \geq 1} \bigg\{\pi \in \SSp(\Cf^\infty(SL_q (2, \Cbb))) \mid \theta_\pi \text{ admits a CB extension to } A(SU_q(2), w^{FL}_\beta) \bigg\}, \nonumber 
\end{align*}
and each subset of the union contains all representations of the form $ \pi = \pi_c \otimes \widehat{\epsilon} $ for $ \pi_c \in \SSp(\Cf^\infty(SU_q(2))) $, where $ \widehat{\epsilon} $ denotes the counit of $ \Dcal(SU_q(2)) $. 
\end{thm}

This result is a quantum analogue of Proposition \ref{prop:the main result 2 of Ludwig et al.} if we regard $ \SSp(A) $ as the ``noncommutative space" which represents the $ * $-algebra $ A $.

\section{Beurling-Fourier algebras of \texorpdfstring{$ q $}{TEXT}-deformations} \label{sec:the generalisation}

In this section we study Beurling-Fourier algebras associated with $ q $-deformations of general compact semisimple Lie groups, and we generalise the main results of \cite{Franz2021} to this setting.

\subsection{The \texorpdfstring{$ q $}{TEXT}-deformation of compact semisimple Lie groups} \label{sec:the q-deformation of compact semisimple Lie groups}

Throughout this section, we adopt the notation of \cite{VoigtYuncken} and freely use the results therein.

Fix $ 0 < q < 1 $. Let $ K $ be a simply-connected compact semisimple Lie group and let $ G = K_\mathbb{C} $ be the complexification of $ K $. 
If $ \kf $ and $ \gf $ denote the respective Lie algebras then $ G $ is the simply-connected semisimple Lie group with Lie algebra $ \mathfrak{g} = \mathfrak{k}_\mathbb{C} $. Fix a maximal torus $ T \subseteq K $ with Lie algebra $ \tf \subseteq \kf $ and a set $ \Sigma = \{\alpha_1, \dots, \alpha_N\} \subseteq (i \tf)^* $ of simple roots with respect to the Cartan subalgebra $ \hf = \tf \otimes \Cbb $ of $ \gf $. We write $ (\cdot,\cdot) $ for the bilinear form on $ \mathfrak{h}^* $ obtained by rescaling the Killing form such that the shortest root $ \alpha $ satisfies $ (\alpha, \alpha) = 2 $. Let $ \{\varpi_1, \dots, \varpi_N\} $ be the set of fundamental weights. Moreover write $ \roots $ and $ \weights $ for the abelian subgroups of $ (i \tf)^* $ 
generated by $ \{\alpha_1, \dots, \alpha_N \}$ and $ \{\varpi_1, \dots, \varpi_N\} $, respectively. The set $ \weights^+ $ of dominant integral weights is the subset of $ \weights $
consisting of all non-negative integer linear combinations of $ \{\varpi_1, \dots, \varpi_N \} $.

Let $ U_q^\Rbb(\kf) $ be the quantized universal enveloping algebra $ U_q(\gf) $ of $ \gf $ equipped with the $ * $-structure corresponding to the compact real form $ \kf $ 
as in \cite[Chapter 4]{VoigtYuncken}. We denote the standard generators of this algebra by
$$
K_\lambda , \;\; E_i , \;\; F_i, \;\; \lambda \in \weights, \; 1\leq i \leq N.
$$
There is a one-to-one correspondence between $ \weights^+ $ and the set of equivalence classes of irreducible integrable $ * $-representations of $ U_q^\Rbb(\kf) $. 
Denoting the representation corresponding to $ \mu \in \weights^+ $ by $ (\pi_\mu, V(\mu)) $, we have 
\begin{equation*} \label{eq:definition of Ccinfty(Kq)}
\Cf^\infty(K_q) = \bigoplus_{\mu \in \weights^+} \End (V(\mu))^*,
\end{equation*}
and the Hopf algebra $ \Cf^\infty(K_q) $ can be equipped with a compact quantum group structure such that the natural pairing
\begin{equation} \label{eq:skew-pairing of Kq}
(\cdot, \cdot): U_q^\Rbb(\kf) \times \Cf^\infty(K_q) \ni (X, (f_\mu)_{\mu \in \weights^+}) \longmapsto \sum_{\mu \in \weights^+} f_\mu (\pi_\mu(X)) \in \Cbb
\end{equation}
becomes a nondegenerate skew-pairing of Hopf $ * $-algebras. The algebra $ \Cf^\infty(K_q) $ is called \textit{the quantized algebra of functions on $ K $}. 
For each $ \mu \in \weights^+ $, fix an orthonormal basis $ (e_i ^\mu )_{1 \leq i \leq n_\mu} $ of $ V(\mu) $ consisting of weight vectors and set
\begin{equation*} \label{eq:basis of Ccinfty(Kq)}
u^\mu _{ij} = \la e_i^\mu | \cdot | e_j ^\mu \ra \in \End (V(\mu))^*, \quad \mu \in \weights^+, \; 1 \leq i,j \leq n_\mu, 
\end{equation*}
where $ n_\mu $ is the dimension of $ V(\mu) $. Then $ (u^\mu _{ij})_{\mu \in \weights^+, 1 \leq i,j \leq n_\mu} $ is a linear basis of $ \Cf^\infty(K_q) $ and the 
matrix $ (u^\mu_{ij})_{1 \leq i,j \leq n_\mu} \in M_{n_\mu}(\Cf^\infty (K_q)) $ is a unitary corepresentation of the Hopf $ * $-algebra $ \Cf^\infty(K_q) $ for 
each $ \mu \in \weights^+ $.

Let $ \Dcal(K_q) $ be the algebraic quantum group dual to $ \Cf^\infty(K_q) $ in the sense of \cite{vanDaele}. Then, as a $ * $-algebra,
\begin{equation*} \label{eq:definition of D(Kq)}
\Dcal(K_q) = \bigoplus_{\mu \in \weights^+} \End(V(\mu)).
\end{equation*}
It follows that we have $ \SSp(\Dcal(K_q)) \cong \weights^+ $ in such a way that $\mu \in \weights^+ $ corresponds to the $ * $-representation $ \pi_\mu $ given by the projection onto the $ \mu $-component. The algebraic multiplier algebra of $ \Dcal(K_q) $ is given by
\begin{equation*} \label{eq:multiplier algebra of D(Kq)}
\Mcal(\Dcal(K_q)) = \prod_{\mu \in \weights^+} \End(V(\mu)), 
\end{equation*}
and $ U_q^\Rbb(\kf) $ embeds into this algebra via the map
\begin{equation*} \label{eq:embedding of Uq(k) into M(D(Kq))}
U_q^\Rbb(\kf) \ni X \longmapsto (\pi_\mu(X))_{\mu} \in \prod_{\mu \in \weights^+} \End(V(\mu)).
\end{equation*}
In particular, the canonical pairing $ \Mcal(\Dcal(K_q)) \otimes \Cf^\infty(K_q) \rightarrow \Cbb $ is a nondegenerate skew-pairing that extends \eqref{eq:skew-pairing of Kq}. 
Accordingly, we also denote it by $ (\cdot, \cdot) $.

Let $ R \in \Mcal(\Dcal(K_q) \otimes \Dcal(K_q)) = \prod_{\lambda, \mu \in \weights^+} \End(V(\lambda)) \otimes \End(V(\mu))$ be the universal $ R $-matrix of $ U_q(\gf) $. We consider the 
$ l $-functionals 
$$
l^\pm: \Cf^\infty(K_q) \rightarrow U_q^\Rbb(\kf)
$$
defined by 
\begin{equation*}\label{eq:definition of l functions}
(l^+(f), g) = (R, g \otimes f), \qquad
(l^-(f), g) = (R^{-1}, f \otimes g).
\end{equation*}
These are Hopf algebra homomorphisms such that
\begin{equation} \label{eq:l functions and involution}
l^\pm(f)^* = l^\mp(f^*), \quad f \in \Cf^\infty(K_q).
\end{equation}

\subsection{The \texorpdfstring{$ q $}{TEXT}-deformation of complex semisimple Lie groups}\label{sec:the q-deformation of complex semisimple Lie groups}

Recall that $ G $ denotes the complexification of the compact semisimple Lie group $ K $. The $ * $-algebra
\begin{equation*} \label{eq:definition of Ccinfty(Gq)}
\Cf_c^\infty(G_q) = \Cf^\infty(K_q) \otimes \Dcal(K_q)
\end{equation*}
can be equipped with an algebraic quantum group structure, representing the \textit{quantum double} of $ K_q $. This algebra is called \textit{the quantized algebra of functions on $ G $}. We have
\begin{equation*} \label{eq:sp Ccinfty(Gq)}
\SSp(\Cf_c^\infty(G_q)) \cong \SSp(\Cf^\infty(K_q)) \times \SSp(\Dcal(K_q)) 
\end{equation*}
via the correspondence
$$
(\pi_c, \pi_\lambda) \longmapsto \pi_c \otimes \pi_\lambda,
$$
where $ \pi_\lambda \in \SSp(\Dcal(K_q)) $ is the $ * $-representation corresponding to $ \lambda \in \weights^+$. Let us note that the irreducible $ * $-representations of the quantized function algebra $ \Cf^\infty(K_q) $ can be described explicitly as well \cite{KorogodskiSoibelman}. 

From the above description it is not immediately apparent why the locally compact quantum group $ G_q $ should be viewed as a complexification of $ K_q $. The rationale comes from deformation theory and the quantum duality principle, see \cite{Drinfeld}, \cite{Gavarini}. More specifically, the $ * $-algebra $ \Dcal(K_q) $ should be interpreted as the algebra of functions $ \Cf_c^\infty(AN_q) $ of the quantisation of $ AN \subset G $, where $ G = KAN $ is the Iwasawa decomposition of $ G $. On the semi-classical level this is reflected by the fact that, as a Poisson Lie group, the group $ G $ is the \textit{classical double} of $ K $, equipped with its standard Poisson structure, see \cite{KorogodskiSoibelman}. 
Note that the Poisson dual of $ K $ is $ AN $. 

Let $ \Cf^\infty(G_q) = \Mcal(\Cf _c ^\infty(G_q)) = \prod_{\mu \in \weights^+} \Cf^\infty(K_q) \otimes \End(V(\mu)) $ be the algebraic multiplier algebra of $ \Cf_c^\infty(G_q) $. Then $ \Cf^\infty(K_q) $ embeds into $ \Cf^\infty(G_q) $ via
\begin{equation*} \label{eq:embedding of Cinfty(Kq) into Cinfty(Gq)}
i: \Cf^\infty(K_q) \ni f \longmapsto f_{(1)} \otimes l^-(f_{(2)}) \in \Cf^\infty(G_q).
\end{equation*}
Let us note that $ i $ is not compatible with the $ * $-structures. 
Every $ \pi \in \SSp(\Cf_c^\infty(G_q)) $ extends to a unital $ * $-representation $ \tilde{\pi}: \Cf^\infty(G_q) \rightarrow \Lbb(\Hcal_\pi) $. 
Therefore, every $ \pi \in \SSp(\Cf_c^\infty(G_q)) $ induces an algebra representation
\begin{equation*} \label{eq:representation of Cinfty(Kq)}
\theta_\pi: \Cf^\infty(K_q) \xrightarrow{i} \Cf^\infty(G_q) \xrightarrow{\tilde{\pi}} \Lbb(\Hcal_\pi).
\end{equation*}
For these representations we shall abbreviate $ v_\pi = v_{\theta_\pi} $, compare \eqref{eq:definition of v_varphi}.

\subsection{The main result} \label{sec:the main result}

Throughout we fix $ 0 < q < 1 $ and consider the quantum groups $ K_q $ and $ G_q $ as above. We start with the following Lemma which is a generalisation of \cite[Proposition~5.3.(1)]{Franz2021}. 

\begin{lem} \label{lem:simplification}
Let $ \pi = \pi_c \otimes \pi_\lambda \in \SSp(\Cf_c^\infty(G_q)) $ for $ \pi_c \in \SSp(\Cf^\infty(K_q)) $ and $ \lambda \in \weights^+ $. Then
$$
v_\pi = v_{\pi_c \otimes \widehat{\epsilon}_\lambda} v_{\epsilon_c \otimes \pi_{\lambda}} 
$$
in $ \prod_{\mu \in \weights^+} \Lbb(\Hcal_\pi) \otimes \Lbb(V(\mu)) $, where $ \epsilon_c $ and $ \widehat{\epsilon}_\lambda $ denote the trivial representations 
of $ \Cf^\infty(K_q) $ and $ \Dcal(K_q) $ on the representation spaces $\Hcal_{\pi_c} $ and $ V(\lambda) $, respectively.
\end{lem}

\begin{proof}
Under the identification $ \Hcal_\pi \cong \Hcal_{\pi_c} \otimes V(\lambda) $, we have
\begin{align*}
v_\pi &= \Big(\sum_{1 \leq i,j \leq n_\mu} \theta_\pi (u^\mu _{ij}) \otimes \omega^\mu_{ij} \Big)_\mu \\
&= \Big(\sum_{1 \leq i, j,k \leq n_\mu} \pi_c (u^\mu _{ik} ) \otimes \pi_\lambda (l^- (u^\mu _{kj})) \otimes \omega ^\mu _{ij} \Big)_\mu \\
&= \Big(\sum_{1 \leq i, k \leq n_\mu} \pi_c(u^\mu_{ik}) \otimes \id \otimes \omega^\mu_{ik} \Big)_\mu 
\Big(\sum_{1 \leq l, j \leq n_\mu} \id \otimes \pi_\lambda(l^-(u^\mu_{lj})) \otimes \omega^\mu_{lj} \Big)_\mu \\
&= v_{\pi_c \otimes \widehat{\epsilon}_\lambda} v_{\epsilon_c \otimes \pi_\lambda}.
\end{align*}
This yields the claim. 
\end{proof}

\begin{cor} \label{cor:simplification}
Given a central weight $ w = (w(\mu))_{\mu \in \Irr(K_q)} $ on the dual of $ K_q $ and $ \pi = \pi_c \otimes \pi_{\lambda} \in \SSp(\Cf_c^\infty(G_q)) $ with $ \pi_c \in \SSp(\Cf^\infty(K_q)) $ and $ \lambda \in \weights^+ $, the map $ \theta_\pi $ extends to a completely bounded representation $ \tilde{\theta}_\pi: A(K_q,w) \rightarrow \Lbb(\Hcal_\pi) $ if and only if
\begin{equation*} \label{eq:CB extension condition for Kq}
\sup_{\mu \in \weights^+} \frac{\| v_{\epsilon_c \otimes \pi_\lambda}^\mu \|}{w(\mu)} < \infty.
\end{equation*}
\end{cor}

\begin{proof}
Note that $ v_{\pi_c \otimes \widehat{\epsilon}_\lambda} = \Big(\sum_{1 \leq i,j \leq n_\mu} \pi_c (u^\mu_{ij}) \otimes \id \otimes \omega^\mu_{ij} \Big)_\mu $ is a unitary 
in $ \prod_{\mu \in \weights^+} \Lbb(\Hcal_\pi) \otimes \Lbb(V(\mu)) $ since $ \pi_c $ is a $ * $-representation and $ (u^\mu_{ij})_{1 \leq i,j \leq n_\mu} $ is a unitary matrix 
in $ M_{n_\mu}(\Cf^\infty(K_q)) $. Therefore, by Lemma~\ref{lem:simplification}, we have
$$
\| v_\pi^\mu \| = \| v^\mu _{\pi_c \otimes \widehat{\epsilon}_\lambda} v^\mu_{\epsilon_c \otimes \pi_\lambda} \| = \| v^\mu_{\epsilon_c \otimes \pi_\lambda} \|
$$
for all $ \mu \in \weights^+ $. The conclusion now follows from Proposition~\ref{prop:CB extension condition}.
\end{proof}

Let us next calculate $ \| v^\mu _{\epsilon_c \otimes \pi_\lambda} \| $ for $ \mu, \lambda \in \weights^+ $.

\begin{lem} \label{lem:computation of v_mu}
Let $ \pi \cong \pi_c \otimes \pi_\lambda \in \SSp(\Cf^\infty (G_q))$ with $ \pi_c \in \SSp(\Cf^\infty(K_q)) $ and $ \lambda \in \weights^+ $. Then, for all $\mu \in \weights^+$,
\begin{equation*} \label{eq:computation of v_mu}
\| v_{\epsilon_c \otimes \pi_\lambda}^\mu \| = \| \pi_\lambda (l^-(u^\mu)) \|,
\end{equation*}
where $ l^-(u^\mu) $ is the matrix with entries $ l^-(u^\mu_{ij}) \in U_q^\Rbb(\kf)$ for $ 1 \leq i,j \leq n_\mu $, and the norm on the right hand side is the operator norm on $ M_{n_\mu}(\Lbb(V(\lambda))) \cong \Lbb(V(\mu)) \otimes \Lbb(V(\lambda)) $.
\end{lem}

\begin{proof}
By definition, we have 
$$
\| v_{\epsilon_c \otimes \pi_\lambda}^\mu \| = \| \sum_{1 \leq i,j \leq n_\mu} \pi_\lambda (l^-(u^\mu_{ij})) \otimes \omega^\mu_{ij} \| 
= \| \pi_\lambda (l^-(u^\mu)) \|
$$
as claimed.
\end{proof}

\begin{lem} \label{lem:estimates on I(u_mu)}
For any $ \lambda, \mu \in \weights^+ $ we have 
\begin{equation*}\label{eq:estimates on I(u_mu)}
\| \pi_\lambda ( l^- (u^\mu ) ) \|= q^{-(\lambda, \mu)}.
\end{equation*}
\end{lem}

\begin{proof}
First note that 
$$
l^- (u^\mu _{ij} ) = (R^{-1} , u^\mu _{ij} \otimes (\; \cdot \;) ) = \big( \la e^{\mu} _{i} | \cdot | e^{\mu} _j \ra \otimes \id \big) (\pi_\mu \otimes \id) (R^{-1})
$$
for $ 1 \leq i,j \leq n_\mu $, and hence 
$$ 
\la e_k ^\lambda, \pi_\lambda( l^- (u^\mu _{ij}))(e_l^\lambda) \ra = \big(\la e^\mu_i \otimes e^\lambda_k | \; \cdot \; | e^\mu_j \otimes e^\lambda_l \ra \big) (\pi_\mu \otimes \pi_\lambda) (R^{-1})
$$
for $1 \leq k,l \leq n_\lambda $. 
Thus we get
$$
\pi_\lambda (l^- (u^\mu )) = (\pi_\mu \otimes \pi_\lambda) (R^{-1}) \in \End(V(\mu)) \otimes \End(V(\lambda)).
$$
Next recall that $ R_{21} = R^* $ where $ R_{21} = \sigma \circ R $, 
which can be seen from \cite[Lemma~4.14]{VoigtYuncken}, using that $ (\widehat{S}^{-1} \otimes \widehat{S}^{-1})(R) = R $. 
It follows that
$$
(\pi_\mu \otimes \pi_\lambda) (R^{-1} )^* = (\pi_\mu \otimes \pi_\lambda) (R_{21} ^{-1}), 
$$
and we get 
\begin{equation}\label{eq:relating the l-representation to the R-representation}
\pi_\lambda (l^- (u^\mu)) \pi_\lambda (l^- (u^\mu))^* = (\pi_\mu \otimes \pi_\lambda) (R^{-1} R_{21} ^{-1} ) = (\pi_\mu \otimes \pi_\lambda) (R_{21} R)^{-1}
\end{equation}
in $\End(V(\mu))\otimes \End(V(\lambda))$. 
Now for $ \nu \subseteq \mu \otimes \lambda $ 
the latter matrix acts on the isotypical component of type $ \nu $ 
inside $ V(\mu) \otimes V(\lambda) $ by  
\begin{equation*}\label{eq:tensor decomposition of the R-matrix}
q^{(\mu, \mu + 2 \rho ) + (\lambda, \lambda + 2 \rho) - (\nu , \nu + 2 \rho)} \id, 
\end{equation*}
compare \cite[Section~8.4.3, Proposition~22]{Klimyk}. 
Let $ n_1, \dots, n_N \in \Nbb $ such that $ \nu = \mu + \lambda - \sum_{1 \leq j \leq N} n_j \alpha_j $. Then we have
\begin{align*}
(\nu, \nu + 2\rho) &= (\mu + \lambda - \sum_{1 \leq j \leq N} n_j \alpha_j, \nu + 2 \rho) \\
&\leq (\mu + \lambda, \nu + 2 \rho) \\
&= (\mu + \lambda , \mu + \lambda + 2 \rho  - \sum_{1 \leq j \leq N} n_j \alpha_j) 
\leq (\mu + \lambda , \mu + \lambda + 2 \rho).
\end{align*}
Since $ 0 < q < 1 $, it follows that the operator norm of $ (\pi_\mu \otimes \pi_\lambda) (R_{21} R)^{-1} $ is given by
$$
\sup_{\nu \subseteq \mu \otimes \lambda} q^{(\mu, \mu+ 2 \rho ) + (\lambda, \lambda+ 2 \rho) - (\nu , \nu + 2 \rho)} = q^{(\mu, \mu+ 2 \rho ) + (\lambda, \lambda+ 2 \rho) - (\mu+\lambda , \mu+\lambda + 2 \rho)} = q^{-2 (\lambda, \mu)}. 
$$
Combining this with \eqref{eq:relating the l-representation to the R-representation} completes the proof. 
\end{proof}

\begin{rmk}\label{rmk:I isomorphism}
Consider the linear map $ I: \Cf^\infty(K_q) \rightarrow U_q(\gf) $ given by 
\begin{equation*} \label{eq:definition of I isomorphism}
I(f) = l^-(f_{(1)}) \widehat{S}(l^+(f_{(2)})), \quad f \in \Cf^\infty (K_q).
\end{equation*}
Then $ I $ is a linear isomorphism from $ \Cf^\infty(K_q) $ onto the locally finite part of $ U_q(\gf) $, compare \cite[Section~3.12]{VoigtYuncken}.

We observe that
\begin{align*}\label{eq:relation between I-matrix and l-matrix}
I(u^\mu _{ij}) = \sum_{1 \leq k \leq n_\mu} l^-(u^\mu _{ik}) \widehat{S}(l^+(u^\mu _{kj})) &= 
\sum_{1 \leq k \leq n_\mu} l^-(u^\mu _{ik}) l^+({S(u^\mu _{kj}})) \nonumber \\
&= \sum_{1 \leq k \leq n_\mu} l^-(u^\mu _{ik}) l^+({u^\mu _{jk}}^*) 
= \sum_{1 \leq k \leq n_\mu} l^- (u^\mu _{ik}) l^- (u^\mu _{jk}) ^*
\end{align*}
for all $ 1 \leq i,j \leq n_\mu$ using that $ l^+ $ is a Hopf algebra homomorphism and \eqref{eq:l functions and involution}. Hence $I(u^\mu) = l^- (u^\mu) l^- (u^\mu)^*$ in $M_{n_\mu} (U_q ^\Rbb (\kf))$. Lemma~\ref{lem:estimates on I(u_mu)} thus shows
$$
\| \pi_\lambda (I(u^\mu)) \| = \| \pi_\lambda (l^- (u^\mu )) \|^2 = q^{-2 (\lambda , \mu)}
$$
for all $\lambda, \mu \in \weights^+$. 
\end{rmk}

We now introduce a family of weights on the dual of $ K_q $ that can be used to detect the irreducible $ * $-representations of $ \Cf_c^\infty(G_q) $.

\begin{prop} \label{prop:weight on the dual of Kq}
For $ \beta \geq 1 $ define $ w_\beta: \weights^+ \rightarrow (0, \infty) $ by 
\begin{equation*} \label{eq:weight on the dual of Kq}
w_\beta (\mu) = \beta^{|\mu|},
\end{equation*}
where $ |\mu| = (\mu , \mu)^{\frac{1}{2}}$. 
Then $ w_\beta = (w_\beta(\mu))_{\mu \in \weights^+} $ is a symmetric central weight on the dual of $ K_q $.
\end{prop}

\begin{proof}
Condition Z1 of Proposition~\ref{prop:central weight on the dual} holds by construction. To check Z2, let $\nu \subseteq \lambda \otimes \mu$. Then, since $\nu$ appears as a weight in the representation $\pi_\lambda \otimes \pi_\mu$, we have $ \nu = \lambda' + \mu' $ for some weights $ \lambda' $ of $ V(\lambda) $ and $ \mu' $ of $ V(\mu) $, respectively. Thus
$$
|\nu| = |\lambda' + \mu'| \leq |\lambda'| + |\mu'| \leq |\lambda| + |\mu|.
$$
This proves $ w_\beta(\nu) = \beta^{|\nu|} \leq \beta^{|\lambda|} \beta^{|\mu|} = w_\beta (\lambda) w_\beta (\mu) $ as required.
Finally, the symmetry of $ w_\beta $ follows immediately from the fact that $ |\mu| = |\overline{\mu}| $ for all $ \mu \in \weights^+ $.
\end{proof}

Combining the above considerations we arrive at the following result. 

\begin{thm} \label{thm:the main result}
Fix $ \beta \geq 1 $. For any $ \pi_c \in \SSp(\Cf^\infty(K_q)) $ we have  
\begin{align}\label{eq:the main result 1}
\Big\{\lambda \in \weights^+ \Big\mid \theta_{\pi_c \otimes \pi_\lambda} \text{ CB-extends to $ A(K_q, w_\beta)$} \Big\} = \Big\{\lambda \in \weights^+ \Big\mid \; q^{- |\lambda|} \leq \beta \Big\}.
\end{align}
Hence 
\begin{align} \label{eq:the main result 2}
\SSp(\Cf_c^\infty(G_q)) &= \bigcup_{\beta \geq 1} \bigg\{\pi \in \SSp(\Cf_c^\infty(G_q)) \mid \begin{array}{clc} \theta_\pi \text{ admits a CB extension to } A(K_q, w_\beta) \end{array} \bigg\}, 
\end{align}
and each set in the union on the right hand side contains all representations of the form $ \pi = \pi_c \otimes \widehat{\epsilon} $ for $ \pi_c \in \SSp(\Cf^\infty(K_q)) $, where $ \widehat{\epsilon} $ denotes the counit of $ \Dcal(K_q) $.  
\end{thm}

\begin{proof}
Combining Lemma ~\ref{lem:computation of v_mu}, Lemma \ref{lem:estimates on I(u_mu)} and the definition of $ w_\beta $ we obtain
\begin{equation*}
\frac{\| v^{\mu} _{\epsilon_c \otimes \pi_\lambda} \|}{w_\beta (\mu)}
= \frac{q^{- ( \lambda, \mu )}}{\beta^{|\mu|}}
\end{equation*}
for all $ \lambda, \mu \in \weights^+ $. Thus, Corollary~\ref{cor:simplification} shows that 
\begin{equation}\label{eq:the key estimates}
\text{$\theta_{\pi_c \otimes \pi_\lambda}$ CB-extends to $ A(K_q , w_\beta) $ if and only if
$\sup_{\mu \in \weights^+} \frac{q^{-(\lambda, \mu)}}{\beta^{|\mu|}} < \infty$}
\end{equation}
for $\lambda \in \weights^+$ and any $ \pi_c \in \SSp(\Cf^\infty(K_q)) $. Since $ q^{-1} > 1 $ we have
$$
q^{-(\lambda, \mu)} \leq q^{-|\lambda| | \mu |}, \quad \lambda , \mu \in \weights^+,
$$
which, in light of \eqref{eq:the key estimates}, proves the inclusion ``$\supseteq $" in \eqref{eq:the main result 1}.

For the reverse inclusion, assume that $ \theta_{\pi_c \otimes \pi_\lambda} $ extends to a completely bounded representation of $ A(K_q , w_\beta) $ for $ \lambda \in \weights^+ $. If there was $ \mu_0 \in \weights^+ $ such that $ (\lambda, \mu_0) > | \mu_0 | \frac{\log \beta}{\log q^{-1}} $, then we would obtain 
$$ 
\Big(\frac{q^{-(\lambda ,m \mu_0)}} {\beta^{|m\mu_0|}}\Big) = \Big(e^{(\lambda, \mu_0 ) \log q^{-1} - |\mu_0 | \log \beta } \Big)^{m} \xrightarrow[m \rightarrow \infty]{} \infty,
$$
which contradicts \eqref{eq:the key estimates}. Thus, we must have $ (\lambda, \mu) \leq |\mu| \frac{\log \beta}{\log q^{-1}} $ for all $ \mu \in \weights^+ $. In particular, setting $ \mu = \lambda $ we get
$$
|\lambda| \leq \frac{\log \beta}{\log q^{-1}},
$$
proving the inclusion ``$\subseteq$" in \eqref{eq:the main result 1}. Finally note that \eqref{eq:the main result 2} is an immediate  consequence of \eqref{eq:the main result 1}.
\end{proof}

The following special case of Theorem \ref{thm:the main result} can be thought of as a quantum version of Eymard duality \cite{Eymard1964} in that the ``noncommutative space" $\SSp(\Cf^\infty(K_q))$ underlying the quantum group $ K_q $ shows up as a set of algebra representations of $ A(K_q) $.

\begin{cor}\label{cor:quantum Eymard's duality?}
We have 
\begin{align*}
\bigg\{\pi \in \SSp(\Cf_c^\infty(G_q)) \mid \begin{array}{clc} \theta_\pi \text{ admits a CB extension to } A(K_q) \end{array} \bigg\} \\
= \big\{ \pi_c \otimes \widehat{\epsilon} \in \SSp(\Cf_c^\infty(G_q) \mid \pi_c \in \SSp(\Cf^\infty(K_q)) \big\}.
\end{align*}
\end{cor}
\begin{proof}
Set $ \beta = 1 $ in the first assertion of Theorem~\ref{thm:the main result}.
\end{proof}

\begin{rmk} \label{rmk:comparison with the work of Franz}
Consider the case $ K = SU(2) $ and identify $\weights^+ \cong \frac{1}{2} \Nbb$. For $ s \in \frac{1}{2} \Nbb $ we have 
$$
|s| = (s,s)^{\frac{1}{2}} 
= \sqrt{2}s,
$$
compare \cite[Section 6.7]{VoigtYuncken}. 
Hence, setting $ \gamma = \beta^{\sqrt{2}} $ we see that $w_\gamma(s) = \beta^{2s} = w_\beta^{FL}(s) $ is the weight considered by Franz and Lee, compare \eqref{eq:weight on the dual of SUq(2)}. If one restricts attention to positive deformation parameters $ q $ it follows that Theorem~\ref{thm:the main result} becomes precisely \cite[Theorem~5.6]{Franz2021} in this case, see Theorem~\ref{thm:the main result of Franz and Lee}.

Our approach yields in fact a slightly stronger result even for $ K = SU(2) $. More precisely, it follows from Lemma~\ref{lem:estimates on I(u_mu)} that the constant $ C_q $ in \cite[Theorem~5.5]{Franz2021} can be chosen to be $ 1 $. 
\end{rmk}

\emph{Acknowledgements}: This work began when H. Lee was visiting C. Voigt at the University of Glasgow, which was supported by the Basic Science Research Program through the National Research Foundation of Korea (NRF) Grant NRF-2022R1A2C1092320.

\printbibliography 

\end{document}